\documentclass[leqno,a4paper,10pt]{article}
\usepackage{amsmath,amsthm,amssymb}
\usepackage[utf8]{inputenc}
\usepackage[T1]{fontenc}

\usepackage{enumerate}
\usepackage{bbm}
\usepackage{todonotes}
\usepackage{mathrsfs}
\usepackage{mathabx}
\usepackage{hyperref}
\usepackage{nicefrac}
\usepackage{tikz}
\usetikzlibrary{matrix}
\usepackage{todonotes}

\usepackage{url}
\makeatletter
\g@addto@macro{\UrlBreaks}{\UrlOrds}
\makeatother
\usepackage[sort&compress,numbers]{natbib}

\providecommand{\noopsort}[1]{} 

\usepackage{xcolor}
\hypersetup
{
	colorlinks = true,
	linkcolor = {red},
	anchorcolor = {black},
	citecolor ={red},
	filecolor ={cyan},
	menucolor= {red},
	runcolor ={cyan - same as file color},
	urlcolor= {red},
}

\newtheorem{Th}{Theorem}[section]
\newtheorem{Prop}[Th]{Proposition}
\newtheorem{Lemma}[Th]{Lemma}

\newtheorem{Cor}[Th]{Corollary}
\theoremstyle{definition}
\newtheorem{Remark}[Th]{Remark}

\newtheorem{Example}{Example}[section]

\newcommand{\beq}{\begin{equation}}
\newcommand{\eeq}{\end{equation}}

\def\scalar(#1,#2){(#1\mid#2)}

\newcommand{\raz}{\mathbbm{1}}

\newcommand{\no}{\#_1}

\newcommand{\cf}{{\cal F}}

\newcommand{\cm}{{\cal M}}

\newcommand{\ot}{\otimes}
\newcommand{\ov}{\overline}
\newcommand{\la}{\lambda}

\newcommand{\R}{{\mathbb{R}}}

\newcommand{\Z}{{\mathbb{Z}}}
\newcommand{\N}{{\mathbb{N}}}

\newcommand{\cL}{{\mathcal{L}}}

\newcommand{\cM}{{\cal M}}
\newcommand{\cF}{{\cal F}}

\newcommand{\mob}{\boldsymbol{\mu}}

\newcommand{\dnu}{d_\nu}

\newcommand{\hnu}{h_\nu}
\newcommand{\Bern}{B_{\nicefrac12, \nicefrac12}}

\newcommand{\bdelta}{\boldsymbol{\delta}}

\title{Hereditary subshifts whose measure of maximal entropy has no Gibbs property}
\author{Joanna Ku\l{}aga-Przymus \and Micha\l{} D. Lema\'nczyk}
\date{27 July 2020}
\begin{document}
\maketitle
\thispagestyle{empty}
\begin{abstract} We show that  the measure of maximal entropy for the hereditary closure of a $\mathscr{B}$-free subshift has the Gibbs property if and only if the Mirsky measure of the subshift is purely atomic. This answers an open question asked by Peckner. Moreover, we show that $\mathscr{B}$ is taut whenever the corresponding Mirsky measure $\nu_\eta$ has full support. This is the converse theorem to a recent result of Keller.
\end{abstract}

\section{Introduction}\label{se intro}
In this paper we consider subshifts $(X,S)$, where $X\subset \{0,1\}^\Z$ and $S$ denotes the left shift. Let $\cM(X,S)$ be the set of probability Borel $S$-invariant measures on~$X$. Recall that $\cM(X,S)$ is compact (and metrizable) in the weak-$\ast$ topology. Let $\cM^e(X,S)\subset \cM(X,S)$ stand for the subset of ergodic measures.		

For $x\in X$, let $\cL(x)$ denote the family of all blocks appearing in $x$. The set $\cL(X):=\bigcup_{x\in X}\cL(x)$ is called the {\em language} of $X$. For any $n\geq 1$, by $\cL_n(X) \subset \cL(X)$ we will denote the subset of blocks of length $n$. When $(X, S)$ is clear from the context we will sometimes abbreviate $\cL = \cL(X)$, $\cL_n = \cL_n(X)$, $\cM = \cM(X, S)$ and $\cM^e = \cM^e(X, S)$. Moreover, we will use words "block" and "word" interchangeably.

			Fix a word $W = \left[w_0 w_1 \ldots w_{n - 1}\right] \in \cL_n$ and $x=(x_i)_{i\in\Z} \in X$. Then:
			\begin{itemize}
				\item  $|W| = n$ will stand for the length of $W$;
				\item $ W[i, j] = \left[w_i\ldots w_j\right]$ for $0\le i \le j \le n - 1$ will denote a subword of $W$;
				\item $x[i,j] = \left[x_i,\dots,x_j\right]$ for $i\leq j$ will denote a subword of $x$;
				\item $\#_1 W = \# \left\{0 \le i \le n - 1  : w_i = 1\right\}$ will be the number of ones in $W$.
			\end{itemize}
Given a block $C\in\cL(X)$, we will often denote by the same letter $C$ the corresponding cylinder set, i.e. $\{x\in X: x[0,|C|-1]=C\}$, for example, $[1]=\{x\in X : x[0]=1\}$.
\paragraph{Hereditary subshifts}
We recall that subshift  $(X, S)$ with language $\cL$ is \textit{hereditary} if
				\begin{equation*}
					W \in \cL,\; W' \le W \Rightarrow W' \in \cL,
				\end{equation*}
				where $\leq$ is to be understood coordinatewise 
			(see e.g. \cite{Ke-Li,Kw} for basic properties and examples of such systems). Moreover, for any subshift $(X, S)$, one can define the \textit{hereditary closure} of $X$, $(\widetilde{X},S)$,  via
				\begin{equation*}
					\widetilde{X}:=\{z\in \{0,1\}^{\Z}:  z\leq x\text{ for some }x\in X\}.
				\end{equation*}	
			It follows immediately that $X$ is hereditary iff  $\widetilde{X} = X$. Examples of hereditary systems include many $\mathscr{B}$-free systems (see below), spacing shifts \cite{MR0321058}, beta shifts (\cite{Renyi}, for the proof of heredity, see~\cite{Kw}), bounded density shifts~\cite{MR3122156} or some shifts of finite type (see Section~\ref{sft}). Most of them are intrinsically ergodic (i.e.\ they have a unique measure of maximal entropy), see below for $\mathscr{B}$-free shifts,~\cite{Cl-Th} for beta shifts and \cite{MR3957218} for a subclass of bounded density shifts (for other listed examples, to our best knowledge, intrinsic ergodicity remains open).

\paragraph{Convolution measures}
			In our examples the measure of maximal entropy will have a special form. Let $Q\colon X \times \{0,1\}^\Z	\rightarrow \widetilde{X}$ be the coordinatewise multiplication:	
				\begin{equation*}\label{definition of coordinatewise multiplication function}
					Q(x, y) = \left(\ldots, x_{-1}y_{-1}, x_0y_0, x_1y_1, \ldots\right)
				\end{equation*}
				for $x=(x_i)_{i\in\Z}\in X$, $y=(y_i)_{i\in\Z}\in \{0,1\}^\Z$.
For any measures $\nu\in\cM(X,S)$ and $\mu \in \cM(\{0, 1\}^\Z,S)$ the \textit{multiplicative convolution of} $\nu$ and $\mu$ is the measure $\nu \ast \mu \in \cM(\widetilde{X})$ given by
				\begin{equation*}\label{definition of convolution of measures}
					\nu \ast \mu = Q_\ast (\nu \otimes \mu),
				\end{equation*}
				where $Q_\ast$ stands for the image of $\nu \otimes \mu$ via $Q$.\footnote{Similar notation will be used for other maps and measures, too.}
			In particular, we will be interested in measures of the form
				\beq\label{definition of kappa being convolution with bernoulli}
					\kappa=\nu\ast \Bern,
				\eeq
			where $\nu\in \cM^{e}(X,S)$ and $\Bern$ stands for the Bernoulli measure on $\{0,1\}^\Z$ with $\Bern([0])=\Bern([1])=\nicefrac{1}{2}$.
			
			It is not hard to see that $\kappa=\nu\ast \Bern$ is of full support as soon as $\nu$ is so. Moreover, $\kappa=\nu\ast \Bern$ is ergodic whenever $\nu$ is ergodic.

\paragraph{Entropy}
		Recall that given a subshift $(X,S)$, its \textit{topological entropy} $h=h(X,S)$ is defined as follows:		 
		 	\begin{equation}\label{definition of topological entropy}
		 		h =\lim_{n\to\infty} \frac{1}{n}\log\left(\left|\cL_n \right|\right) = \inf_{n\in\N} \frac{1}{n}\log\left(\left|\cL_n \right|\right).\footnote{Throughout the paper function $\log$ always stands for~$\log_2$.}
		 	\end{equation}
		Similarly, for any $\nu\in\cM$ the \textit{measure entropy} $h_\nu$ can be computed as
			\begin{equation}\label{definition of measure entropy}
				h_\nu =  \lim_{n\to\infty} \frac{1}{n}h_\nu (\cL_n) = \inf_{n\in\N} \frac{1}{n}h_\nu(\cL_n),
			\end{equation}
		where $\hnu(\cL_n) = -\sum_{W \in \cL_n} \nu(W)\log\left(\nu(W)\right)$ denotes the Shannon entropy with respect to the partition of $X$ into blocks given by the elements of $\cL_n$. It is well-known that $\hnu$ and $h$ are related via the following variational principle:
			\begin{equation}\label{variational principle for entropy}
				h = \sup_{\nu\in\cM} \hnu.
			\end{equation}
		Finally, for any probability vector $\boldsymbol{p} = (p_1, \ldots, p_n)$ let us denote by $H(\boldsymbol{p})$ the Shannon entropy
			\begin{equation}\label{definition of Shannon entropy general case}
				H(\boldsymbol{p}) = -\sum_{i = 1}^{n} p_i \log(p_i).
			\end{equation}
		In particular, if $\boldsymbol{p} = (p, 1 - p)$ we will write
			\begin{equation}\label{definition of Shannon entropy special case n eq 2}
				H(p) = H(\boldsymbol{p}) = -p\log(p) - (1 - p)\log(1 - p).
			\end{equation}
		Notice that $H(p)$ is symetric with respect to $1/2$, strictly increasing (decreasing) on $[0, 1/2]$ ($[1/2, 1]$).

\paragraph{Gibbs property}
A measure $\kappa\in \cM^e(X,S)$ is {said to have a}  \textit{Gibbs property} if there exists a constant $a>0$ such that
				\beq\label{gibbs}
					\kappa(C)\geq a\cdot 2^{-|C|h(X,S)}
				\eeq			
			for all blocks $C\in \cL(X)$ having \textbf{positive} $\kappa$-measure. A motivation to study this property is that in many natural situations, like sofic systems~\cite{We0} or systems enjoying particular specification properties and beyond (see \cite{Cl-Th,Cl-Th1} and the references therein), there is a unique measure of maximal entropy and it enjoys the Gibbs property or a weakening of it. More than that, by a result of B.~Weiss \cite{We0}, if $\kappa$ satisfies the Gibbs property and is a measure of maximal entropy, then $(X,S)$ is intrinsically ergodic. We will be interested in examples, where~\eqref{gibbs} fails, but the system under consideration remains intrinsically ergodic.
			This yields classes of positive entropy intrinsically ergodic systems different from many known so far.

The notion of Gibbs measures comes from statistical physics~\cite{MR1747792,Lanford} and it corresponds to the idea of equilibrium states of complicated physical systems. They turned out to be an interesting object also from the point of view of dynamics and have played an important role in ergodic theory (see, e.g.,~\cite{MR2423393,MR0399421}). Given a finite alphabet $\mathsf{A}$ and a (H\"older) continuous function $\varphi\colon \mathsf{A}^\Z \to \R$ (often referred to as a \textit{potential}) and a subshift $X\subset \mathsf{A}^\Z$, a measure $\mu_\varphi\in \cm(X,S)$ is called a \textit{Gibbs measure} for $\varphi$, whenever there exist constants $P=P(\varphi,X)\geq 0$ and $c=c(\varphi,X)>0$ such that
		$$
		c^{-1}\leq \frac{\mu_\varphi(x[0,n-1])}{2^{\sum_{k=0}^{n-1}\varphi(S^kx)-nP}} \leq c
		$$
(constant $P$ above is so called topological pressure of $\varphi$; for $\varphi\equiv 0$, we have $P=h(X,S)$). Clearly, if $\mu$ is a Gibbs measure corresponding to $\varphi\equiv 0$ then the lower bound~\eqref{gibbs} holds. The Gibbs property \eqref{gibbs} of equilibrium states corresponding to more general potentials for $\mathscr{B}$-free subshifts will be considered in a forthcoming paper.

\paragraph{$\mathscr{B}$-free systems}			
Our direct interest in looking at the Gibbs property is motivated by the class of
so called $\mathscr{B}$-{\em free systems} $(X_{\eta},S)$. For $\mathscr{B}\subset \N\setminus \{1\}$, let $\mathcal{F}_\mathscr{B}$ be the set of \textit{$\mathscr{B}$-free integers}, i.e.\ numbers with no divisors in $\mathscr{B}$. Its complement (the set of multiples of $\mathscr{B}$) will be denoted by $\mathcal{M}_\mathscr{B}$. We will tacitly assume that $\mathscr{B}$ is primitive, i.e.\ whenever $b\divides b'$ for some $b,b'\in\mathscr{B}$ then $b=b'$. Such sets were studied already in the 30’s by number theorists~\cite{MR1512943,zbMATH03014412,Davenport:1933aa,Da-Er,Davenport1936,MR1574879}, whereas the dynamical approach is more recent. Sarnak in 2010~\cite{sarnak-lectures} proposed to study the dynamical systems related to the M\"obius function $\mob$ and its square $\mob^2$ (the latter is nothing but the characteristic function of $\mathcal{F}_\mathscr{B}$, with $\mathscr{B}$ equal to the set of all squares of primes). For a general $\mathscr{B}$, we set $\eta:=\raz_{\cf_\mathscr{B}}$ and consider the subshift $(X_\eta,S)$, where
			 $$X_{\eta}:=\ov{\{S^n\eta:n\in\Z\}}.$$
In the most classical case (so-called Erd\"os), one assumes that $\mathscr{B}$ is infinite, the elements of $\mathscr{B}$ are mutually coprime and $\sum_{b\in\mathscr{B}}1/b<+\infty$ (in particular, if $\mathscr{B}=\{p^2 : p\in\mathscr{P}\}$, where $\mathscr{P}$ stands for the set of all primes, we speak of the \textit{square-free} system).
Ergodic and topological properties of the corresponding $\mathscr{B}$-free systems $(X_\eta,S)$ were studied in~\cite{Ab-Le-Ru,Ko-Ku-Kw,Ku-Le-We,MR3612882,MR3430278}. General $\mathscr{B}$-free systems, defined as above, were thoroughly studied in~\cite{MR3803141}. A continuation of this research and various further natural generalizations can be found in~\cite{Baake:2015aa,MR3947636,Ke,Keller:2017aa,MR3296562}.

Central role in studying properties of $\mathscr{B}$-free systems is played by so called \textit{Mirsky measure}, denoted by $\nu_\eta$. In the classical, i.e.\ Erd\"os, case, it can be defined by the frequencies of blocks appearing in $\eta$, first studied by Mirsky~\cite{MR0028334,MR0021566}. In other words, $\eta$ is a generic point for $\nu_\eta$:
			 \begin{equation}\label{generic point}
			 \lim_{N\to\infty}\frac{1}{N}\sum_{n\leq N}f(S^n\eta)= \int f\, d\nu_\eta \text{ for all continuous functions $f$ on $X_\eta$}.
			 \end{equation}
For an arbitrary $\mathscr{B}\subset \N\setminus \{1\}$, $\eta$ may fail to be a generic point~\cite{MR1512943}. However, Davenport and Erd\"os~\cite{Da-Er} proved that the logarithmic density  of $\mathcal{F}_\mathscr{B}$, i.e.\ $\bdelta(\mathcal{F}_\mathscr{B})=\lim_{n\to\infty}\frac{1}{\log N}\sum_{a\in \mathcal{F}_\mathscr{B},1\leq a\leq N}\frac{1}{a}$, always exists. Moreover, the logarithmic density of $\mathcal{F}_\mathscr{B}$ is equal to its lower density. A consequence of this fact, proved in~\cite{MR3803141}, is that $\eta$ is quasi-generic for a shift-invariant measure (denoted by $\nu_\eta$) along any subsequence $(N_k)_{k\geq 1}$ realizing the lower density of $\mathcal{F}_\mathscr{B}$, i.e.\ $(N_k)_{k\geq 1}$ is such that $\underline{d}(\mathcal{F}_\mathscr{B})=\lim_{k\to\infty}\frac{1}{N_k}|\mathcal{F}_\mathscr{B}\cap [1,N_k]|$ and the convergence in~\eqref{generic point} takes place along this subsequence. This result can serve for our purposes as a definition of the Mirsky measure (both, in~\cite{Ab-Le-Ru} and~\cite{MR3803141} the definiton of the Mirsky measure is more algebraic: $\nu_\eta$ is the image of Haar measure of the smallest closed subgroup of $\prod_{b\in\mathscr{B}}\Z/b\Z$ containing $(1,1,\dots)$ via a certain map).

In the Erd\"os case, the Mirsky measure $\nu_\eta$ has full support and $X_\eta$ is hereditary~\cite{Ab-Le-Ru}. Moreover, each such $(X_{\eta},S)$ is transitive, has positive entropy and displays properties similar to the full shift, see also \cite{Ku-Le-We1,Ko-Ku-Kw}. Besides, $(X_{\eta},S)$ is intrinsically ergodic and the convolution $\nu_{\eta}\ast \Bern$ is the measure of maximal entropy of $(X_{\eta},S)$ \cite{Ku-Le-We}.

In the square-free case, Peckner in~\cite{MR3430278} showed that the measure $\nu_{\eta}\ast \Bern\in \cM^e(X_{\eta},S)$ has no Gibbs property. His proof however used essentially some non-trivial number-theoretic facts concerning the
set of squares of primes. Peckner asked whether the absence of Gibbs property for the above convolution measure is characteristic for all $\mathscr{B}$-free systems.\footnote{More precisely, Peckner was interested in the Erd\"os case.}

\paragraph{Zoo of densities}			
Our main tool will be the following four notions of \textit{density}: $d$ and $D$ of topological nature and their measure-theoretic counterparts $d_\nu$ and $D_\nu$. Fix $(X, S)$ and let
			\begin{equation*}\label{definition of lower density}
				d := \sup_{\mu \in \cM}\mu\left(\left[1\right]\right),
			\end{equation*}			
			\begin{equation}\label{definition of upper density}
				D := \lim_{n\rightarrow \infty} \frac{\max_{W\in \mathcal{L}_n} \#_1 W }{n}=\sup_{x\in X}\overline{d}_B(\{n\in \N : x[n]=1\}),
			\end{equation}
			where $\overline{d}_B$ stands for the upper Banach density. Moreover, for $\nu\in\cM(X,S)$, let
			\begin{equation*}\label{definition of lower measure density}
				d_\nu := \nu\left(\left[1\right]\right),
			\end{equation*}
			\begin{equation*}\label{definition of upper measure density}
				D_\nu := \lim_{n\rightarrow \infty} \frac{\max_{W\in \mathcal{L}_n,\; \nu\left(W\right) >0} \#_1 W }{n}.
			\end{equation*}
Notice that both $D$ and $D_\nu$ are well defined since the sequences $\left(\max_{W\in \mathcal{L}_n} \#_1 W\right)_n$ and $\left(\max_{W\in \mathcal{L}_n,\; \nu\left(W\right) >0} \#_1 W\right)_n$ are subadditive. In particular, we can replace $\lim$'s by $\inf$'s.

To understand better what the above quantities mean and what are the relations between them, we need the following folklore result:\footnote{We would like to thank one of the referees for suggesting this way of formulating our results.}
\begin{Th}\label{prop11}
Let $(X,T)$ be a topological dynamical system. Let $x\in X$ and let $A\subset X$ be a clopen set. Then
\[
\overline{d}_B(\{n\in\N : T^nx\in A\})=\sup\{ \nu(A) :  \nu\in\mathcal{M} \text{ such that }\nu(\overline{\{T^nx : n\in\Z\}})=1\}.
\]
\end{Th}
We include below the proof for the reader convenience, as we could not find a direct reference. Let us recall first however a result of a very similar flavour. 
\begin{Th}[Theorem 2.6 in~\cite{Bergelson_2008}]\label{BerDow}
Let $(X,T)$ be a topological dynamical system and let $x\in X$. The following conditions are equivalent:
\begin{itemize}
\item
the point $x\in X$ is essentially recurrent, i.e.\ for any neighbohood $U_x$ of $x$ the set of visits $\{n\in\Z : T^nx\in U_x\}$ has positive upper Banach density;
\item
the orbit closure of $x$ under $T$ is measure saturated, i.e.\ for every nonempty open subset $U$ of the orbit closure of $x$, there exists  an invariant measure $\mu$ with $\mu(U)>0$.
\end{itemize}
\end{Th}
In fact, the proofs of Theorem~\ref{prop11} and of Theorem~\ref{BerDow} go along the same lines (while it does not seem to us that one of the results implies the other).
\begin{proof}[Proof of Theorem~\ref{prop11}]
We will show first that there exists $\nu$ such that
\[
\nu(\overline{\{T^nx : n\in\Z\}})=1 \text{ and }\overline{d}_B(\{n\in\N : T^nx\in A\})=\nu(A).
\]
Fix $x\in X$ and let $(m_k)\subset \N$ be an increasing sequence such that
\begin{equation}\label{zbieznosc}
\overline{d}_B(\{n\in\N : T^nx\in A\})=\lim_{k\to\infty} \frac1k \# ([m_k,m_k+k-1]\cap \{n\in\N : T^nx\in A\}).
\end{equation}
Let $x^{(k)}:=T^{m_k}x$ and
\[
\nu_k:=\frac1k\sum_{i=0}^{k-1}\delta_{T^ix^{(k)}}.
\]
We may assume without loss of generality that $\nu_k\to\nu$ weakly. Notice that $\nu$ is $T$-invariant and is concentrated on the orbit closure of $x$ under $T$. Rewriting~\eqref{zbieznosc} and using the fact that $\raz_A$ is continuous, we obtain
\[
\overline{d}_B(\{n\in\N : T^nx\in A\})=\lim_{k\to \infty}\frac1k\sum_{i=0}^k\raz_A(T^ix^{(k)})=\lim_{k\to\infty}\int \raz_A\, d\nu_k=\nu(A).
\]

We will now show that $\nu(A)$ for $\nu\in\mathcal{M} \text{ such that }\nu(\overline{\{T^nx : n\in\Z\}})=1\}$ cannot exceed $\overline{d}_B(\{n\in\N : T^nx\in A\})$. Using the ergodic decomposition, it is clear that it suffices to prove it for $\nu$ ergodic. For any ergodic $\nu$ that is concentrated on the orbit closure of $x$ under $T$, we can find a generic point $y$ in the orbit closure of $x$. In particular, one can find $m_k$ such that 
\[
\{0\leq i\leq k-1 : T^{m_k+i}x \in A\}= \{0\leq i\leq k-1 : T^{i}y \in A\}
\]
(recall that $A$ is clopen).
It follows that
\begin{align*}
\overline{d}_B(\{n\in\N : T^nx\in A\})\geq \lim_{k\to\infty}\frac1k \# ([m_k,m_k+k-1] \cap \{n\in\N : T^nx \in A\})\\
=\lim_{k\to\infty}\frac1k\#([0,k-1]\cap \{n\in\N : T^ny\in A\})=\nu(A)
\end{align*}
which completes the proof.
\end{proof}
Taking into account the additional supremum over all $x\in X$ and applying the above proposition to $A=[1]$, we immediately obtain that
\begin{equation*}\label{dD}
D=\sup_{\mu\in\mathcal{M}}\mu([1])=d.
\end{equation*}
Moreover, using~\eqref{definition of upper density} and Proposition~\ref{prop11} yields
\begin{align*}
D_\nu&=\lim_{n\to\infty}\frac{\max_{W\in \mathcal{L}_n(\textrm{supp}(\nu))}\#_1W}{n}\\
&=\sup_{x\in \textrm{supp}(\nu)}\overline{d}_B(\{n\in \N : x[n]=1\})=\sup\{\mu([1]) : \mu(\textrm{supp}(\nu))=1\}.
\end{align*}
It follows immediately that 
\begin{equation}\label{main inequalities}
d_\nu\leq D_\nu\leq D=d.
\end{equation}
Moreover, we have the following \textit{variational principle}:
			\begin{equation}\label{variational principle for densities}
				D = \sup_{\nu \in \cM} D_\nu.
			\end{equation}

We will call a measure $\nu\in\cm(X,S)$ a \textit{maximal density measure} if $d_\nu = d$. Furthermore, if $D_\nu = D$, we will say that $\nu$ is \textit{ones-saturated}.

\begin{Remark}\label{existence}
Notice that a measure of maximal density always exists. {Indeed, since $f=\raz_{[1]}$ is continuous, it follows that the map $\nu\mapsto \nu([1])$ is continuous.}
\end{Remark}
\begin{Remark}
Clearly, each measure of full support is ones-saturated. Moreover, it follows from~\eqref{main inequalities} that also each measure of maximal density is ones-saturated, {whence, by Remark~\ref{existence},} a ones-saturated measure always exists.
\end{Remark}

\begin{Remark}
It is a classical fact in the theory of cut-and-project sets that for any $\mathscr{B}$, the Mirsky measure $\nu_\eta$ is a measure of maximal density for $(X_\eta,S)$ (see e.g. Theorem 4 and Corollary 4 in~\cite{Ke-Ri}, cf.\ also Chapter 7 in~\cite{Ba}. To obtain a maximal density measure without full support, consider a $\mathscr{B}$-free system that is not taut and take its Mirsky measure (see Corollary~\ref{max dens not full support}).
\end{Remark}

\begin{Example}
Let us see that the inequalities in~\eqref{main inequalities} can be sharp:
\begin{enumerate}[(A)]
\item
To obtain $d_\nu < D_\nu$, consider the full shift and $\nu = B_{p, q}$ being the Bernoulli measure on $\{0,1\}^\Z$ with $p=\nu([1])\in (0,1)$ ($q=1-p=\nu([0])$). Then $D_\nu = 1$ whilst $d_\nu = p$.
\item\label{B}
We have already noticed that $D_\nu<D$ implies that $\nu$ is not fully supported.  Consider the full shift with $\nu =\nicefrac{1}{2} (\delta_{(\ldots 01.0101\ldots)} + \delta_{(\ldots10.1010\ldots)})$. In this case $D_\nu = \nicefrac12$ and $D = 1$.
\item
Consider the full shift with measure $\kappa=\nu\ast B_{p,q}$, where $\nu$ is as in \eqref{B}. Then
	$
	d_\kappa=\nicefrac{p}{2},\ D_\kappa=\nicefrac12,\ D=1.
	$
\end{enumerate}

\end{Example}

\paragraph{Entropy vs density}

\begin{Example}
We will now give examples showing possible relations between $d$ and $h$ (without putting any extra assumptions on $(X,S)$).
			\begin{enumerate}[(A)]
				\item For the full shift we have $d=h=1$.
				\item For each zero entropy subshift with an invariant measure different from $\delta_{\ldots0.00\ldots}$, we have $0=h<d$.
				\item\label{C} Fix $0< p< 1/2$ (and think about $p$ to be very close to $1/2$). Using the Jewett-Krieger theorem we may find a uniquely ergodic subshift $(X,S,\nu)$  measure-theoretically isomorphic to the Bernoulli shift $(\{0,1\}^{\Z},B_{p,q},S)$.
				First, we will show that (recall \eqref{definition of Shannon entropy special case n eq 2})
					\begin{equation}\label{entropy calculated}
						h(X, S) = H(p).
					\end{equation}
				Indeed, by \eqref{variational principle for entropy}, we have $h(X,S) = \hnu$, {which} gives
					$$h(X, S) = \hnu = h_{B_{p, q}} = H(p).$$
				Now, we turn to estimating $d(X, S)$. We will show that
					\begin{equation}\label{density estimated}
						p \le d(X,S) \le 1 - p.
					\end{equation}
				The partition $\alpha = ([0],[1])$ is a generating one, so
					$$H(\nu\left([1]\right)) = H(\nu\left([0]\right)) \ge \hnu(S,\alpha) = h(X, S) = H(p).$$
				Due to the shape of function $H$, it follows that
					$$p \le \nu\left([1]\right) = \dnu = d(X, S)  \le 1 - p.$$
				If we take $0< p < 1/2$ such that $1 - p < H(p)$, \eqref{entropy calculated} with \eqref{density estimated} imply $d <h$.
			\end{enumerate}
\end{Example}

		If we take $p$ very close to $1/2$ in \eqref{C} above, we can make $d$ as close to $1/2$ and $h$ as close to $1$ as we like. Therefore,  there seems to be no continuity of $h$ with respect to $d$. Nevertheless, we have the following "continuity" lemma:
		
			\begin{Lemma}
				Suppose that $d \leq \nicefrac12$. Then
					\begin{equation}\label{inequality entropy less than density}
						h \le  H(d).
					\end{equation}
					In particular, if $d \rightarrow 0$ then $h \rightarrow 0$.
					\begin{proof}
					Let $\mu\in \mathcal{M}$. Then
					\[
						h_\mu(X)=h_\mu(\{0,1\}^\Z)=\lim_{n\to\infty}\frac1n h_\mu(\{0,1\}^n) \leq H(\mu([1]))
					\]
					as the above limit in non-increasing. Moreover, since $\mu([1])\leq d\leq \nicefrac12$ and since $H$ on $(0,\nicefrac12)$ is increasing, we obtain
					\[
					h_\mu(X)\leq H(d).
					\]
					Finally, the variational principle allows us to replace $h_\mu(X)$ on the left hand side by the topological entropy $h$.
					\end{proof}
			\end{Lemma}

			For an arbitrary subshift $(X, S)$ one may also wonder what are the relations between $d = d(X, S)$, $h = h(X, S)$ and $\widetilde{d} = d(\widetilde{X}, S)$, $\widetilde{h} = h(\widetilde{X}, S)$. Using the right hand side of \eqref{main inequalities}, one obtains immediately that
				\begin{equation*}
					d = \widetilde{d}.
				\end{equation*}
				Moreover, 			
				\begin{equation*}
					\widetilde{h} \ge h.
				\end{equation*}
			Due to the fact that in a hereditary subshift we can "downgrade" ones to zeros it is clear that
				\begin{equation}\label{nier}
					\widetilde{h} \ge d.
				\end{equation}
			We summarize the  above in the following chain of inequalities:
				\begin{equation*}
					d = \widetilde{d} \le \widetilde{h} \text{ and } h \le \widetilde{h}.
				\end{equation*}
			One can have both, $d=\widetilde{h}$ (see the remainder of this section) and $d<\widetilde{h}$ (see Section~\ref{sft}).

\paragraph{Results}
Our main tool to prove the absence of Gibbs property for hereditary subshifts is the following technical result:
				\begin{Th}\label{main theorem}
				Fix $(X,S)$ and suppose that $\nu\in\cM^e(X,S)$ is ones-saturated and non-atomic. If $d=\widetilde{h}$ then $\kappa=\nu\ast \Bern$ does not have Gibbs property.
				\end{Th}					

Recall that in general, $\mathscr{B}$-free systems $(X_\eta,S)$ need not be hereditary.  However, by \cite{MR3803141}, $(\widetilde{X}_{\eta},S)$ is intrinsically ergodic and the convolution $\kappa=\nu_{\eta}\ast \Bern$ is the unique measure of maximal entropy in $(\widetilde{X}_{\eta},S)$. Moreover, the Mirsky measure is a measure of maximal density for $(X_\eta,S)$, in particular, it is ones-saturated.\footnote{For the details, see Section~\ref{uzupelnienia B-free}.}  Finally, we have $d=\widetilde{h}$ (see Proposition K in~\cite{MR3803141}).

As an immediate consequence of these remarks and Theorem~\ref{main theorem}, we obtain the positive answer to Peckner's question:
\begin{Cor}\label{c:bwolne}
Let $\mathscr{B}\subset \N\setminus \{1\}$. Suppose that the Mirsky measure $\nu_\eta$ is not atomic. Then the measure of maximal entropy of $(\widetilde{X}_\eta,S)$ {does not have} the Gibbs property.
\end{Cor}

It follows from~\cite{MR3803141} that when $\mathscr{B}$ is finite\footnote{In this case $\eta$ is periodic, so the Mirsky measure is discrete.} then $\widetilde{X}_{\eta}$ is sofic, so its measure of maximal entropy has the Gibbs property, {see also Remark~\ref{r:skon} for a direct argument}. A description of those $\mathscr{B}$ for which the Mirsky measure has an atom (which seems to be of independent interest) is given in Proposition~\ref{atomowa}.

As a ``byproduct'', we also prove several results on $\mathscr{B}$-free systems that are of independent interest. In particular, we prove the following (which is the converse theorem to a recent result by Keller~\cite{Ke}, for the details, see Section~\ref{uzupelnienia B-free}):
\begin{Cor}\label{converse}
Let $\mathscr{B}\subset \N\setminus \{1\}$. If the Mirsky measure $\nu_\eta$ is of full support $X_\eta$ then $\mathscr{B}$ is taut.
\end{Cor}

Theorem~\ref{main theorem} goes beyond the $\mathscr{B}$-free context. For example, if $(X,S)$ is of zero topological entropy, it follows by Lemma 2.2.16 in~\cite{Ku-Le-We} that $d\geq \widetilde{h}$, {which}, together with~\eqref{nier}, yields $d=\widetilde{h}$. Thus, as a consequence of Theorem~\ref{main theorem}, we obtain the following:
\begin{Cor}\label{c:zeroen}
If $(X,S)$ is uniquely ergodic with $h(X,S)=0$, then $\nu\ast \Bern$ has no Gibbs property whenever the unique invariant measure $\nu$ is non-atomic.
\end{Cor}

In~\cite{Ku-Le-We} also so-called Sturmian sequences are discussed.\footnote{Sturmian sequences yield strictly ergodic models of irrational rotations. For more information, we refer the reader to~\cite{Ku-Le-We}.} In particular, it is proved that the hereditary closure {of the system given by any} Sturmian sequence yields an intrinsically ergodic system whose measure of maximal entropy is of the form $\nu\ast \Bern$. Moreover, in this case we also have $d=\widetilde{h}$. Using again Theorem~\ref{main theorem}, we obtain:
\begin{Cor}\label{c:sturmian}
If $(\widetilde{X},S)$ is a Sturmian hereditary system then its measure of maximal entropy has no Gibbs property.
\end{Cor}

\section{Absence of Gibbs property}\label{dowody2}

For the proof of Theorem~\ref{main theorem} we need some simple lemmas.
			\begin{Lemma}\label{l:gi1}
Let $\nu\in \cM(X,S)$. Then for $\kappa=\nu\ast \Bern$, we have
\beq\label{e3}
						\kappa(C)=\sum_{\cL(X)\ni C'\geq C}\nu(C')\cdot 2^{-\no C'}
\eeq
for each $C\in \cL(\widetilde{X})$.
\end{Lemma}
\begin{proof}
For {$C\in \cL_{\ell}(\widetilde{X})$}, we have
$
\kappa(C)=\left(\nu\ot \Bern\right)(E)$,
where
$$
E:=\{(y,z)\in Y\times \{0,1\}^{\Z}: (y\cdot z)[0,\ell-1]=C\}=\bigcup_{\cL(Y)\ni C'\geq C,D\geq C, C'\cdot D=C} C'\times D.
$$
Moreover the sets $C'\times D$ in the above union are pairwise disjoint. The values of each $D$  are determined on the support of $C'$ and are arbitrary on its complement (if $C'[i]=1$ then $D[i]=C[i]$). Since
$$
\left(\nu\ot \Bern\right)\left(\bigcup_{D\geq C,C'\cdot D=C}C'\times D\right)=\nu(C')\cdot 2^{-\no C' },$$
the result follows.
\end{proof}
\begin{Remark}
	Notice that $\kappa=\nu\ast \Bern$ is descreasing in the sense that for any two words of {length} $n$, $W_1$, $W_2$  such that $W_1 \le W_2$, we have
				\begin{equation}\label{increasing property of kappa}
					\kappa\left(W_1\right) \geq \kappa\left(W_2\right).
				\end{equation}
\end{Remark}

			We will say that a block $C \in \cL(X)$ is \textit{ones-maximal} if
				\begin{equation}\label{definition of block being ones-maximal}
					\no C  = \max_{W\in \mathcal{\cL}_{|C|}(X)} \no W.
				\end{equation}
				Analogously, for any measure $\nu \in \cM(X)$, we will say that a block $C \in \cL(X)$ is $\nu$-\textit{ones-maximal} if
				\begin{equation}\label{definition of block being measure-ones-maximal}
					\no C  = \max_{W\in \mathcal{\cL}_{|C|}(X), \; \nu(W) > 0} \no W.
				\end{equation}
\begin{Remark}				
				Notice that if $C$ is $\nu$-ones-maximal (or ones-maximal) then \eqref{e3} reduces to
				\begin{equation}\label{formula for maximal blocks}
					\nu\ast \Bern(C)= \nu(C)\cdot 2^{-\no C}.
				\end{equation}
				{Note that the above formula also works for maximal blocks (in the sense of the coordinatewise order).}
\end{Remark}

			\begin{Lemma}\label{lemat atom}
				Let $\nu\in \cM(X,S)$ and $a > 0$. Suppose that there is a sequence of blocks $C_n$ such that $|C_n| \nearrow \infty$ and $\nu(C_n)\geq a$. Then there exists $(n_k)$ such that $\bigcap_{k\geq 1}C_{n_k}\neq \emptyset$. Moreover, we have $\nu(\{x\})\geq a$ for $\{x\}=\bigcap_{k\geq 1}C_{n_k}$.
			\end{Lemma}
				\begin{proof}
					Notice that for any $k\geq 1$, there exists $B\in \cL_k(X)$ such that for infinitely many $n\in\N$, we have $C_n[0,k-1]=B$. Now, we apply a diagonal procedure to find $(n_k)$. Moreover, $\nu(\bigcap_{k\geq 1} C_{n_k})=\nu(\{x\})\geq a$.
				\end{proof}

				\begin{proof}[Proof of Theorem~\ref{main theorem}]
					For $n \in \N$, let $C_n\in\mathcal{L}_n$ be $\nu$-ones-maximal. Define $o_n := \no C_n$. For $\kappa=\nu\ast \Bern$, we have
					$$
					\kappa(C_n)\cdot 2^{n\widetilde{h}}=\nu(C_n)\cdot 2^{n\widetilde{h}-o_n}.
					$$
					Using~\eqref{main inequalities}, we have $d = D = D_\nu \le o_n/n$. Therefore, by the assumption that $\widetilde{h}=d$, we obtain $n\widetilde{h}-o_n\leq 0$. It follows by Lemma~\ref{lemat atom} that $\kappa$ fails to satisfy the Gibbs property.
				\end{proof}

\section{Mirsky measure}\label{uzupelnienia B-free}

\subsection{Topological support}
An important role in the theory of $\mathscr{B}$-free systems is played by the notion of tautness. Recall that for $\mathscr{B}\subset \N\setminus \{1\}$, we set $\mathcal{M}_\mathscr{B}:=\bigcup_{b\in\mathscr{B}}b\Z$. We say that $\mathscr{B}\subset \N\setminus \{1\}$ is \textit{taut}~\cite{Ha} if for any $b\in \mathscr{B}$, we have the following inequality between logarithmic densities: ${\boldsymbol{\delta}}(\mathcal{M}_\mathscr{B})>\boldsymbol{\delta}(\mathcal{M}_{\mathscr{B}\setminus\{b\}})$. It was shown in~\cite{MR3803141} (see Corollary 2.31 therein) that the tautness of $\mathscr{B}$ implies the following:
\begin{equation}\label{behnowy}
\text{if $\boldsymbol{\delta}(\mathcal{M}_{\mathscr{B}\cup\{a\}})=\boldsymbol{\delta}(\mathcal{M}_\mathscr{B})$ then $a\in\mathcal{M}_\mathscr{B}$.}
\end{equation}
Moreover, for each $\mathscr{B}\subset\N\setminus\{1\}$ there exists a unique taut $\mathscr{B}'$ such that 
\begin{equation}\label{t:taut}
\cf_{\mathscr{B}'}\subset\cf_{\mathscr{B}},\  \nu_{\eta}=\nu_{\eta'}\text{ and }\mathcal{M}(\widetilde{X}_\eta,S)=\mathcal{M}(\widetilde{X}_{\eta'},S),\text{ where }\eta'=\eta(\mathscr{B}')
\end{equation}
(see Thm.~C and 4.5 in \cite{MR3803141}).

\begin{Remark}
As a consequence of Theorem~\ref{twierdzenie keller} and~\eqref{t:taut}, we obtain another proof of the fact that the Mirsky measure is ones-saturated (both, in $X_\eta$ and in $\widetilde{X}_\eta$). Indeed, for a taut $\mathscr{B}$ we use Theorem~\ref{twierdzenie keller} and for other sets $\mathscr{B}$, we combine it with~\eqref{t:taut}.
\end{Remark}

Recall that $\mathscr{B}$ is said to be \textit{Behrend}~\cite{Ha} if $\bdelta(\mathcal{F}_\mathscr{B})=0$. Each infinite set of primes whose sums of reciprocals is infinite is Behrend (see (0.69) in~\cite{MR1414678}). Take $a,r\in\N$ with $\gcd(a,r)=1$. Dirichlet proved that $a\Z+r$ contains infinitely many primes and $\sum_{p\in (a\Z+r)\cap \mathscr{P}}1/p=+\infty$. Thus
\begin{equation}\label{behrenda}
a\Z+r \text{ is Behrend whenever }\gcd(a,r)=1.
\end{equation}

Recently, Keller proved the following:
\begin{Th}[\cite{Ke}]\label{twierdzenie keller}
If $\mathscr{B}$ is taut then $\nu_\eta$ has full support in $X_\eta$.
\end{Th}
The main ingredient in the proof of the converse result (i.e.\ Corollary~\ref{converse}) is the following. Suppose that $\mathscr{B}$ is not taut and let $\mathscr{B}'$ be the corresponding taut set, as in~\eqref{t:taut}. Then 
\begin{equation}\label{relacja}
{X}_{\eta'} \subsetneq {X}_\eta.
\end{equation}
Suppose for a moment that we have already proved~\eqref{relacja}.
\begin{proof}[Proof of Corollary~\ref{converse}]
Suppose that $\mathscr{B}$ is not taut. Let $\mathscr{B}'$ be the corresponding taut set, as in~\eqref{t:taut}. Then $\nu_{\eta}={\nu_{\eta'}}$. Moreover, by Theorem~\ref{twierdzenie keller}, $X_{\eta'}$ is the support of ${\nu_{\eta'}}$. It follows immediately from~\eqref{relacja} that the support of $\nu_\eta$ (equal to $X_{\eta'}$) is not full.
\end{proof}

\begin{proof}[Proof of~\eqref{relacja}] We will prove first that $X_{\eta'}\subset X_\eta$. By Theorem~\ref{twierdzenie keller}, $\nu_{\eta'}$ is of full support $X_{\eta'}$, i.e.\ each block appearing in $\eta'$ is of positive $\nu_{\eta'}$-measure. By~\eqref{t:taut}, we have $\nu_\eta=\nu_{\eta'}$, i.e.\ each block appearing in $\eta'$ is of positive $\nu_\eta$-measure. Since $\eta$ is a quasi-generic point for $\nu_\eta$, each block of positive $\nu_\eta$-measure appears on $\eta$. Therefore, each block appearing in $\eta'$ appears also on $\eta$, which gives $X_{\eta'}\subset X_\eta$.

Suppose now that ${X}_{\eta'} = {X}_\eta$. In particular, we have $\eta\in{X}_{\eta'}\subset X_{\mathscr{B}'}$.\footnote{{Recall that $X_\mathscr{B'}$ is the so-called admissible subshift, i.e.\ $x\in X_\mathscr{B'}$ iff $|\text{supp }x\bmod b'|<b'$ for each $b'\in \mathscr{B}'$ (by $\text{supp }x$ we understand the set $\{n\in\Z : x[n]\neq 0\}$). See~\cite{sarnak-lectures}, for the square-free case.}} Therefore, for each $b'\in\mathscr{B}'$, there exists $1\leq r'\leq b'$ such that
	$\mathcal{F}_\mathscr{B} \cap (b'\Z+r')=\emptyset$,
	i.e.\ $b'\Z+r'\subset \mathcal{M}_\mathscr{B}$. Let $d=\gcd(b',r')$. For $b'':=b'/d$, $r'':=r'/d$, we have
	$$d(b''\Z+r'')\subset \mathcal{M}_\mathscr{B}{\subset \cm_{\mathscr{B}'}}.$$
It follows by this and by~\eqref{behrenda} that
$ \bdelta(\mathcal{M}_{\mathscr{B}'})=
\bdelta(\mathcal{M}_{\mathscr{B}'\cup\{d\}})$.
	By~\eqref{behnowy}, we obtain $d\in\mathcal{M}_{\mathscr{B'}}$. Hence, there exists $b'''\in\mathscr{B}'$ such that $b''' \divides d$, i.e.\ we have $b''' \divides d \divides b'$. Thus, by the primitivity of $\mathscr{B}'$, we obtain $b'''=d=b'$. Therefore, $r'=b'$ and we conclude that $b'\Z\subset \mathcal{M}_\mathscr{B}$. Since $b'\in\mathscr{B}'$ was arbitrary, it follows that $\mathcal{M}_\mathscr{B}=\mathcal{M}_{\mathscr{B}'}$. Now, it remains to use the primitivity of $\mathscr{B}$ and $\mathscr{B}'$ to conclude that $\mathscr{B}=\mathscr{B}'$. This yields a contradiction and completes the proof.
\end{proof}

As an immediate consequence of Corollary~\ref{converse}, we obtain the following:
\begin{Cor}\label{max dens not full support}
Suppose that $\mathscr{B}$ is not taut. Then {the Mirsky measure $\nu_{\eta}$} has maximal density in $X_\eta$, but does not have full support.
\end{Cor}
\begin{Remark}
In~\cite{Dy}, it was shown that $\mathscr{B}$-free systems that are minimal (equivalently, Toeplitz~\cite{MR3803141}), are necessarily taut. Notice that this also follows immediately from Corollary~\ref{converse}, as in minimal systems all invariant measures have full support.
\end{Remark}

\paragraph{Atomic Mirsky measure}

\subsection{Atoms}
 We will now describe all sets $\mathscr{B}$ for which the Mirsky measure is atomic.

\begin{Prop}\label{p:atom}
The Mirsky measure $\nu_{\eta}$ is atomic if and only if the taut  set $\mathscr{B}'$ given by~\eqref{t:taut} is finite.
\end{Prop}

\begin{proof}
Clearly, if $\mathscr{B}'$ is finite then the corresponding Mirsky measure is atomic. We will prove now the other implication. In view of~\eqref{t:taut}, we can assume that $\mathscr{B}$ itself is taut, and we need to prove that in this case $\mathscr{B}$ is finite. But if $\mathscr{B}$ is taut then by Theorem~F in~\cite{MR3803141} the measure-theoretic dynamical system $(X_{\eta},\nu_{\eta},S)$ is isomorphic
to {a} rotation on a certain compact Abelian group considered with Haar measure. However, Haar measure has an atom if and only if the group is finite. Since the group is given by the inverse limit of
cyclic groups $\Z/{\rm lcm}(\{b\in \mathscr{B}: b\leq K\})$, $K\geq 1$, $\mathscr{B}$ itself is finite.
\end{proof}
\begin{Cor}\label{atomowa}
The Mirsky measure $\nu_{\eta}$ is atomic if and only if for some $k,\ell\geq 1$
\begin{equation}\label{postac}
\mathscr{B}=c_1\mathscr{B}_1\cup \dots \cup c_k\mathscr{B}_k \cup \{c'_1,\dots,c'_\ell\},
\end{equation}
with $\mathscr{B}_1,\dots,\mathscr{B}_k$ being Behrend.
\end{Cor}

\begin{proof}
Let $\mathscr{B}'$ be as in~\eqref{t:taut}. It follows by the construction of the taut set $\mathscr{B}'$ in Section 4.2 in~\cite{MR3803141} that either
\begin{equation}\label{bprim}
\mathscr{B}'=(\mathscr{B}\setminus (c_1\Z\cup\dots {\cup} c_n\Z)) \cup \{c_1,\dots,c_n\}
\end{equation}
and
\begin{equation}\label{setb}
\mathscr{B}=(\mathscr{B}\setminus (c_1\Z\cup\dots {\cup} c_n\Z)) \cup (c_1\mathscr{B}_1\cup \dots\cup c_n\mathscr{B}_n)
\end{equation}
for some $n\geq 1$ and some Behrend sets $\mathscr{B}_1,\dots, \mathscr{B}_n$ or
$$
\mathscr{B}'=(\mathscr{B}\setminus  \bigcup_{n\geq 1}c_n\Z) \cup \{c_n : n\geq 1\}
$$
and
$$
\mathscr{B}=(\mathscr{B}\setminus  \bigcup_{n\geq 1}c_n\Z) \cup \bigcup_{n\geq 1}c_n\mathscr{B}_n
$$
for some Behrend sets $\mathscr{B}_n,n\geq 1$.

The finiteness of $\mathscr{B}'$ means that \eqref{bprim} and \eqref{setb} hold for some $n\geq 1$. In particular, the set
$\mathscr{B}\setminus (c_1\Z\cup\dots \cup c_n\Z)$ is finite, i.e.\ \eqref{postac} holds.
\end{proof}

		\section{Gibbs property vs entropy}\label{gibbs1}
			By~\eqref{variational principle for entropy}, notice that if $h =0$ then $h_\kappa = 0$ for any $\kappa \in \cM$. In general, if $h>0$, it is hard to say for which $\kappa$ we have $h_\kappa > 0$. The situation changes if we assume that the $\kappa$ has the Gibbs property (and full support).  	
				\begin{Prop}
					Suppose that $\kappa \in \cM$ has full support and satisfies the Gibbs property~\eqref{gibbs}. Then
						\begin{equation}\label{inequality entropy of gibbs measure is positive if entropy is pos}
							h > 0 \Rightarrow h_\kappa \ge ah.\footnote{When $h=0$ then clearly $h_\kappa\ge ah$ also holds.}
						\end{equation}
				\end{Prop}		
					 \begin{proof}
					 	Let $\ell_n = \left|\cL_n\right|$. Notice that \eqref{definition of topological entropy} implies that $\frac{\log \ell_n}{n} \ge h$ for any $n \in \N$, i.e.~we have
						\begin{equation}\label{toto}
						\ell_n\geq 2^{nh}.
						\end{equation}
						
						 Moreover, {the} function $x \mapsto -x \log x$ is increasing for  $x\le 1/2$. Due to the full support of $\kappa$ and the Gibbs property~{\eqref{gibbs},  we obtain}
\begin{equation}\label{ee1}
 -\sum_{W \in \cL_n} \kappa\left(W\right)\log\kappa\left(W\right)\ge \sum_{W \in \cL_n,\; \kappa\left(W\right) \le 1/2} a 2^{-nh}\left[nh - \log\left(a\right)\right].
\end{equation}
Since only one atom of the partition given by $\cL_n$ can have the measure larger than $\nicefrac12$, it follows that
\begin{equation}\label{ee2}
 \sum_{W \in \cL_n,\; \kappa\left(W\right) \le 1/2} a 2^{-nh}\left[nh - \log\left(a\right)\right] \ge(\ell_n - 1) a 2^{-nh}\left[nh - \log\left(a\right)\right].
\end{equation}
Now, we apply~\eqref{toto} to get
							\begin{equation}
								\begin{split}\label{ee3}
(\ell_n - 1) a 2^{-nh}\left[nh - \log\left(a\right)\right] &\ge a \left(2^{nh} - 1\right)2^{-nh}\left[nh - \log\left(a\right)\right] \\
									& = a \left(1 - 2^{-nh}\right)\left(nh - \log\left(a\right)\right).
								\end{split}
							\end{equation}
							Combining~\eqref{ee1}, \eqref{ee2} and \eqref{ee3}, we obtain
							\begin{equation*}
								h_\kappa {\xleftarrow[n\to\infty]{}} \frac{-\sum_{W \in \cL_n} \kappa\left(W\right)\log\left(\kappa\left(W\right)\right)}{n} \ge  a \left(1 - 2^{-nh}\right)\left(h - \frac{\log\left(a\right)}{n}\right) {\xrightarrow[n\to\infty]{}}  a h			
							\end{equation*}
						and the result follows. 	
					\end{proof} 	

				\begin{Remark}
					If $X = \{0,1\}^{\Z}$ and $\kappa$ is an ergodic measure of full support with the Gibbs property then
						\begin{equation*}
							\kappa = \Bern.
						\end{equation*}
						Indeed, the inequality in \eqref{gibbs} can be rewritten as $\kappa(C) \geq a\cdot \Bern(C)$ for each block $C$. We obtain $\kappa = \Bern$  by the ergodicity of $\kappa$ and $\Bern$.	
				\end{Remark}
			
Note also that if $h = 0$ then $\kappa$ cannot have the Gibbs property unless it is purely atomic.

{Finally, let us give the following observation concerning the rate of convergence in the formula for topological entropy.
			\begin{Lemma}
				Suppose we can find $\nu \in \cM(X,S)$ such that $\kappa=\nu\ast \Bern$ satisfies Gibbs property~\eqref{gibbs} and has full support then
				\begin{equation*} \label{gibbs necessary condition for conv 2 rate of convergence of top entr}
					0 \le \frac{\log\left(\ell_n\right)}{n} - h \le \frac{1}{n}\log\left(\frac{1}{a}\right)\textnormal{, for every } n \in \N.
				\end{equation*}
				where $\ell_n = |\cL_n(X)|$.
				\end{Lemma}
			\begin{proof}
	It follows from the ``decreasing property''~\eqref{increasing property of kappa} that for any $n\in\N$ there exists a maximal word (in the sense of the coordinatewise order) $W_n^{min}\in\cL_n(X)$ such that for every $W_n\in\cL_n(X)$, we have $\kappa\left(W_n\right) \ge \kappa\left(W_n^{min}\right)$. Then
					$$\kappa\left(W_n^{min}\right) \le \frac{1}{\ell_n}.$$
				Now, taking advantage of the Gibbs property, we get
				$$
						a 2^{-hn}  \le\kappa(W_n^{min})= \nu\left(W_n^{min}\right)2^{-|W_n^{min}|}  \le 2^{-\log\left(\ell_n\right)}.
				$$
				Thus
				$$
				a \le 2^{-n\left[\frac{\log\left(\ell_n\right)}{n} - h\right]}
				$$
				and finally
				$$
				n\left[\frac{\log\left(\ell_n\right)}{n} - h\right] \le \log\left(\frac{1}{a}\right),
				$$
				which gives the desired rate of convergence.
			\end{proof}	}

\section{Hereditary sofic systems}\label{gibbs2}
Fix a finite alphabet $\mathsf{A}$ and let $(G,\mathsf{L})$ be a labeled graph, i.e.\ $G$ is a graph with edge set $E$ and the labeling $\mathsf{L}\colon E\to \mathsf{A}$. Then $X\subset \mathsf{A}^\Z$ arising by reading the labels along the paths on $G$ is called \textit{sofic} (this term was coined by Weiss~\cite{We0} and there are several equivalent ways to define sofic subshifts, see also~\cite{Li-Ma}). For us, $\mathsf{A}=\{0,1\}$. 

Notice that for a sofic subshift $X\subset \{0,1\}^\Z$, the subshift $\widetilde{X}$ is also sofic. Indeed, take a corresponding labeled graph $(G,\mathsf{L})$ for $X$ and define $(\widetilde{G},\widetilde{\mathsf{L}}) $ as follows: for each edge in $G$ labeled with $1$ add an extra edge between the same vertices and label it with $0$. Clearly, the subshift resulting by reading the labels along the paths in the new graph is nothing but $\widetilde{X}$. Recall also that a finite union of sofic shifts remains sofic (to see this, it suffices to consider the corresponding graphs and take their disjoint union).

	\begin{Remark}
		It was shown in~\cite{MR3803141} that that for each finite $\mathscr{B}\subset \mathbb{N}\setminus\{1\}$, both $\widetilde{X}_\eta$
 and $X_\mathscr{B}$ are sofic. A simpler way to prove this is to notice that if $\mathscr{B}$ is finite then $\eta$ is periodic. This gives immediately that $X_\eta$ is sofic and by the above discussion, also $\widetilde{X}_\eta$ is sofic. Moreover, $X_\mathscr{B}$ is a finite union of the following form:
 	$$
	X_\mathscr{B}=\bigcup_{b\in \mathscr{B}}\bigcup_{0\leq r_b \leq b-1}X_{(r_b : b\in\mathscr{B})},
	$$
	where $x\in X_{(r_b : b\in\mathscr{B})}$ iff $(\text{supp }x \bmod b)\cap (b\Z+r_b)=\emptyset$ for each $b\in\mathscr{B}$. Notice also that $X_{(r_b : b\in\mathscr{B})}$ is the hereditary closure of the subshift generated by the periodic point $x_{(r_b : b\in\mathscr{B})}$ whose support equals $\Z \setminus (\bigcup_{b\in\mathscr{B}}(b\Z+r_b))$. Thus, we can apply here the same argument as for $\widetilde{X}_\eta$. 
	 \end{Remark}

		\begin{Cor}
			Suppose that $\nu\in\cM^e(X,S)$ is atomic. Then $\kappa=\nu\ast\Bern$ has the Gibbs property.
		\end{Cor}
		\begin{proof}
			Since $\nu$ is atomic, it follows immediately that $\nu$ is concentrated on a finite orbit, i.e.\ there exists $x_0\in X$ and $k\geq 1$ with $S^kx_0=x_0$ and we have
		$$
			\nu=\frac{1}{k}(\delta_{x_0}+\delta_{Sx_0}+\dots +\delta_{S^{k-1}x_0}).
		$$
		I.e., we may assume that $X=\{x_0,Sx_0,\dots,S^{k-1}x_0\}$.
		It follows from Section~3.2.1 in~\cite{Ku-Le-We} that the measure of maximal entropy for $(\widetilde{X},S)$ is of the form
		$$
		\kappa=\nu\ast\Bern.
		$$
		Moreover, we have $d=\widetilde{h}$. Now, $(\widetilde{X},S)$ as the hereditary closure of finite (i.e.\ sofic) subshift is sofic. Therefore, its measure of maximal entropy has the Gibbs property.
		\end{proof}

{\begin{Remark}\label{r:skon}
As a matter of fact, if $x_0\in\{0,1\}^{\Z}$ is periodic of period $k\geq1$ and $X=\{S^jx_0\colon j=0,\ldots, k-1\}$ then $h(\widetilde{X},S)=d(X,S)=d$, $(\widetilde{X},S)$ is intrinsically ergodic with $\kappa=\nu\ast\Bern$ the measure of maximal entropy. If $n\geq k$ then we have precisely $k$ blocks (in $X$) of length $n$ of $\nu$-measure $1/k$. By the monotonicity~\eqref{increasing property of kappa}, we need to check~\eqref{gibbs} for maximal blocks and for such, by~\eqref{formula for maximal blocks}, we obtain
$$
\kappa(B)=\nu(B)2^{-\no B}\geq \frac1k 2^{-nd}=\frac1k 2^{-nh(\widetilde{X},S)},$$
so $\kappa$ satisfies the Gibbs property.
\end{Remark}}

\section{Subshifts of finite type (SFTs)}\label{sft}

In this section we give examples of subshifts with $d<\widetilde{h}$.

Given a family $\cF\subset\bigcup_{i=1}^\infty\{0,1\}^i$ of blocks, by $X_\cF$ we denote the set of all $x\in\{0,1\}^{\Z}$ such that no block from $\cF$ appears in $x$ (hence, $\cF\cap \cL(X_\cF)=\emptyset$).
A subshift $(X,S)$ is said to be of \textit{finite type} (or Markov) if $X=X_\cF$ for a certain {finite} family of blocks.
			\begin{Remark}\label{FTher} {
Note that if $\cF$ satisfies: $C\in \cF, C'\geq C$ $\Rightarrow$ $C'\in\cF$, then $(X_{\cF},S)$ is hereditary.}
			\end{Remark}
We will make use of some facts from the theory of SFTs given in~\cite{Li-Ma}.
				\begin{Example}
					Consider the {\em golden mean} subshift $X=X_{\{11\}}$. By Remark~\ref{FTher}, $X$ is hereditary. Moreover,
						$$h=\log\frac{1+\sqrt5}2,$$
						see e.g.\ Example 4.1.4 in \cite{Li-Ma} and $d=\nicefrac12$
					(consider $\ldots010101\ldots \in X$ and recall~\eqref{main inequalities}).
				\end{Example}
			However, not all SFTs are hereditary. Now, we will present a SFT that is not hereditary and we have:			
				\begin{equation}\label{SFT exemple non heriditary comparing d and h}
					h < d \text{ and } d=\widetilde{d}<\widetilde{h}
				\end{equation}
				\begin{Example}\label{e:prz2} Consider $\cF=\{00,111\}$ and $X = X_\cF$. We claim that \eqref{SFT exemple non heriditary comparing d and h} is valid. 
				
				We will show first that $h < d$. Note that $\cF':=\{000,001,100,111\}$ is the full list of forbidden blocks of length $3$ and $X_\cF=X_{\cF'}$. Now, the admissible blocks in $X_\cF$ of length $2$ are $11$, $10$ and $01$. Hence, the adjacent matrix $A$ for this subshift is given by
					
					\[
					A =
						\begin{bmatrix}
							 0 & 1 & 0\\
							 0 & 0 & 1\\
							 1 & 1 & 0\\
						 \end{bmatrix}
					\]
					
				\noindent and since $A^4$ has all entries positive, $A$ is aperiodic, that is, $X_{\cF'}$ is irreducible. It follows that $h=\log \la$, where $\la$ is the Perron-Frobenius eigenvalue of $A$. Since the characteristic polynomial equals $t^3-t-1$, we get $\la \approx 1.32$ and
					$$h \approx \log(1.32) \approx 0.4.$$
				Moroever,
					$d = 2/3$
				(consider $x = \ldots 011.011011 \ldots \in X_\cF$), which gives $h < d$.
	
				Now, we turn to the proof of $\widetilde{d}< \widetilde{h}$. The crucial observation is that
					\beq\label{STF subset STF}
						Y := X_{\{111,1001\}}\subset \widetilde{X}_{\cF}.
					\eeq
				
				\noindent Assume for a moment that \eqref{STF subset STF} is true. Then, we have $\widetilde{h} \ge h(Y, S)$, so in order to show $\widetilde{h} > \widetilde{d}$, it is enough to bound $h(Y, S)$ from below. We claim that
					\begin{equation}\label{entropy of Y}
						h(Y, S) \approx 0.76.
					\end{equation}
				In order to see \eqref{entropy of Y}, notice that $3$-admissible blocks in $X_{\{111,1001\}}$ are
					$$000, 100, 010, 001, 110, 101, 011.$$
				Hence, the adjacent matrix equals
					\[
						A =
						\begin{bmatrix}
							1 & 0 & 0 & 1 & 0 & 0 & 0\\
							1 & 0 & 0 & 0 & 0 & 0 & 0\\
							0 & 1 & 0 & 0 & 0 & 1 & 0\\
							0 & 0 & 1 & 0 & 0 & 0 & 1\\
							0 & 1 & 0 & 0 & 0 & 1 & 0\\
							0 & 0 & 1 & 0 & 0 & 0 & 1\\
							0 & 0 & 0 & 0 & 1 & 0 & 0\\
						\end{bmatrix}.
					\]
				Now $A^7>0$, so $A$ is aperiodic. It remains to calculate $\log \la $, where $\la$ is the Perron-Frobenius eigenvalue of $A$, which is approximately $0.76$.

				If remains to prove \eqref{STF subset STF}. For $y\in Y$, we need to find $x\in X$ with $y\leq x$ (coordinatewise). We begin by setting $x:=y$. Now, suppose that somewhere on $x$ we see block of the form
				\begin{equation}\label{form}
				B=1\underbrace{00\ldots0}_{\ell}1.
				\end{equation}
				By the definition of $Y$, either $\ell=1$ or $\ell\geq 3$. If $\ell=1$, we do nothing. If $\ell\geq 3$ and is even, we replace $B$ by $A=1\underbrace{01\ldots10}_{\ell}1$. If $\ell\geq 3$ and is odd, we replace $B$ by $A=1\underbrace{0110101010\ldots1010}_{\ell}1$. We apply this procedure to all occurences of blocks of the form~\eqref{form}. It is easy to see that as a result, we obtain a point $x$ with the desired properties.
			\end{Example}

\subsection*{Acknowledgements}
We would like to thank the referees for comments that helped to improve the presentation of our results. We would also like to thank Mariusz Lema\'{n}czyk for helpful discussions. We also thank Dominik Kwietniak for remarks on spacing shifts and Vitaly Bergelson for directing our attention to his joint paper with Tomasz Downarowicz~\cite{Bergelson_2008}. Research of J.\ Ku\l{}aga-Przymus is supported by the National Science Center, Poland, grant UMO-2019/33/B/ST1/00364. Research of M.D.~Lema\'{n}czyk is supported by the National Science Center, Poland, grant no.\ 2015/18/E/ST1/00214.

\small
\bibliographystyle{siam}
\providecommand{\noopsort}[1]{} 

\allowdisplaybreaks
\small

\bibliography{gibbs}

\bigskip
\footnotesize
\noindent
Joanna Ku\l aga-Przymus\\
\textsc{Faculty of Mathematics and Computer Science, Nicolaus Copernicus University, Chopina 12/18, 87-100 Toru\'{n}, Poland}\par\nopagebreak
\noindent
\href{mailto:joanna.kulaga@gmail.com}
{\texttt{joanna.kulaga@gmail.com}}

\medskip

\noindent
Micha\l{} D. Lema\'{n}czyk \\
\textsc{Faculty of Mathematics, Informatics and Mechanics, Warsaw University, Stefana Banacha 2, 02-097 Warsaw, Poland}\par\nopagebreak
\noindent
\href{mailto:m.lemanczyk@mimuw.edu.pl}
{\texttt{m.lemanczyk@mimuw.edu.pl}}

\end{document}